\documentclass{amsart}
\usepackage[a4paper]{geometry}
\usepackage{graphicx}
\usepackage{color}
\usepackage{mathtools}
\usepackage{tcolorbox}
\usepackage{mathptmx}     
\usepackage{amsfonts} 
\usepackage{amssymb} 
\usepackage{enumitem}
 \usepackage{amsmath}
\usepackage{amscd}
\usepackage{amsthm}
\usepackage{stmaryrd}

\def\Xint#1{\mathchoice
{\XXint\displaystyle\textstyle{#1}}%
{\XXint\textstyle\scriptstyle{#1}}%
{\XXint\scriptstyle\scriptscriptstyle{#1}}%
{\XXint\scriptscriptstyle\scriptscriptstyle{#1}}%
\!\int}
\def\XXint#1#2#3{{\setbox0=\hbox{$#1{#2#3}{\int}$ }
\vcenter{\hbox{$#2#3$ }}\kern-.6\wd0}}

\def\dashint{\Xint-}

\newtheorem{lemma}{Lemma}
\newtheorem{remark}{Remark}

\newtheoremstyle{theostyle}
  {7mm}
  {7mm}
  {\slshape}
  {10pt}
  {\bfseries}
  {:}
  {3pt}
  {}
\theoremstyle{theostyle}
\newtheorem{theorem}{Theorem}

\newcommand{\grad}{\nabla}

\newcommand{\xb}{\bar{x}}
\newcommand{\tb}{\bar{t}}
\newcommand{\fb}{\bar{f}}

\renewcommand{\div}{\,\text{div}\,}

\newcommand{\rb}{\overline \bar{\rho}}

\newcommand{\T}{\mathbb{T}}

\newcommand{\lb}{\llbracket}
\renewcommand{\rb}{\rrbracket}
\renewcommand{\i}{\mbox{\rm i}}
\newcommand{\e}{\epsilon}
\renewcommand{\bar}{\overline}
\numberwithin{equation}{section}

\usepackage{tcolorbox}
\tcbuselibrary{theorems}

\newtcbtheorem[number within=section]{mytheo}{Theorem}%
{colback=green!5,colframe=green!35!black,fonttitle=\rmfamily}{th}

\begin{document}

\title{Existence of weak solutions for the kinetic models of motion of myxobacteria with alignment and reversals}

\begin{abstract}{In this paper, we consider three non-linear kinetic partial differential equations that emerge in the modeling of motion of rod-shaped cells such as myxobacteria. This motion is characterized by nematic alignment with neighboring cells, orientation reversals from cell polarity switching, orientation diffusion, and transport driven by chemotaxis. Our primary contribution lies in establishing the existence of weak solutions for these equations. Our analytical approach is based on the application of the classical averaging lemma from the kinetic theory, augmented by a novel version where the transport operator is substituted with a uni-directional diffusion operator.}
\end{abstract}

\author{Patrick Murphy}
\address{Department of Mathematics and Statistics\\
San Jose State University}

\author{Oleg Igoshin}
\address{Department of Bioengineering\\
Rice University}

\author{Misha Perepelitsa}
\address{Department of Mathematics\\
University of Houston}
\email{mperepel@central.uh.edu}

\author{Ilya Timofeyev}

\address{Department of Mathematics\\
University of Houston}

\date{\today}

\maketitle

\section{Introduction}
Many rod-shaped bacteria are known to move on surfaces and self-organize into various spatial patterns due to mechanical interactions and biochemical signaling between cells. For instance, myxobacteria - bacteria studied for their multicellular interactions can form dynamic patterns of clusters, streams, swirls, and aggregates.
Motivated by the problem of modeling the motion of a single layer of myxobacteria on a flat surface, we consider a mean-filed equation for $f(x,\theta,t)$-the probability distribution functions for bacteria position $x\in\mathbb{R}^2,$ and orientation angle of bacteria longer axis, $\theta\in(-\pi,\pi]:$
\begin{equation}
\label{intro:FK1}
\partial_t f{}+{}e(\theta)\cdot\grad f{}=-{}\partial_{\theta}(f\Psi( v_f)){}+{}\partial^2_{\theta\theta}f -f + \tau_{\pi}f,
\end{equation}
where $e(\theta)=(\cos\theta,\sin\theta),$ as bacteria primary move along their long axis.
\begin{equation}
\label{intro:vf}
v_f(x,t){}={}\int_{-\pi}^{\pi}f(x,\theta_1,t)\sin(2(\theta_1-\theta))\,d\theta_1,
\end{equation}
 $\Psi$ is Lipschitz continuous and bounded, and
$\tau_\pi$ is the translation operator in $\theta$ variable:
\[
\tau_\pi f(x,\theta,t){}={}f(x,\theta+\pi,t),
\]  
The right-hand side of equation \eqref{intro:FK1}
accounts for changes in the angle (passive part) that combine nematic alignment, orientation diffusion, and cell reversals to the opposite direction.  The left-hand side represents the active part: transport in the direction of orientation of the cell's longer axis. The model is based on the kinetic model of Peruani et al. \cite{Peruani2008} for the motion of self-propelled rods with nematic alignment described by Maier-Saupe interaction potential from the theory of liquid crystals.

Alongside equation \eqref{intro:FK1} we consider (i) a kinetic model obtained in the asymptotic ``fast reversal" regime (diffusion limit)
and (ii) a model where the reversals are biased by a concentration of local field $s$, e.g. chemotactic signal.
In the first  case, the kinetic equation is
\begin{equation}
\label{intro:FK2}
\partial_t f{}-{}\div \left(e(\theta)\otimes e(\theta)\grad f\right){}=-{}\partial_\theta(f \Psi(v_f)) {}+{}\partial^2_{\theta\theta}f.
\end{equation}
In this equation, the ``active" part is the spacial diffusion, with the diffusion matrix  $e(\theta)\otimes e(\theta),$ which results in diffusion only in the direction of $e(\theta)$  In this model, the kinetic function $f$ is nematically symmetric: $\tau_\pi f{}={}f.$

In the second case, with $s(x,t)$ denoting the concentration of a chemoattractant, the equation is given by
\begin{equation}
\label{intro:FK3}
\partial_t f + 2\div(  [e(\theta)\cdot\grad s] e(\theta) f){}-{}\div (e(\theta)\otimes e(\theta)\grad f){} =- {}\partial_\theta(f \Psi(v_f)) {}+{}\partial^2_{\theta\theta}f.
\end{equation} As in the classical Keller-Segel model, chemotaxis is expressed as a transport term to the ``active" part of the kinetic equation.

In this paper, we focus on proving the global (in time) existence of weak solutions to Cauchy problem for equations \eqref{intro:FK1}, \eqref{intro:FK2}, and \eqref{intro:FK3} (with given $s(x,t)$), under the assumption that the rate of alignment $\Psi(v_f)$ is a bounded function. 

The equations listed above have unique classical solutions locally in time for any $\Psi\in C^1(\mathbb{R}).$ This can be proved by, for example, by a suitable operator-splitting method. It is, however, unknown if and how the classical solutions blow up in finite time. Such analysis is complicated by the fact that the only known Lyapunov functional for these equations is the total mass. On the other hand, when the rate of alignment is bounded,
the orientational diffusion is sufficient to control the alignment term and there is \textit {a priori} estimate of solutions in $L^2$ norm (equation \eqref{eq:energy1}). This is the case we consider in this paper and the weak solutions we construct have the regularity ascertained by such ``energy" estimate.

The main difficulty of constructing weak solutions is the nonlinearity of the alignment in which the interaction operator $v_f$ is localized to a point. This situation arises in an asymptotic regime of the vanishing ratio of cell length $l$ to a macroscopic scale $L$ (domain size) because the orientational alignment occurs when the cells collide, and thus its range is proportional to $l,$ see section \ref{sec:derive}. This property makes the equations similar to the Boltzmann and Bhatnagar-Gross-Krook (BGK) equations of gas dynamics, where the interaction operators are also localized. Consequently, we look for the analytical approaches developed for such equations and try to apply them here. The key tool in our analysis is the theory kinetic averaging (DiPerna and Lions \cite{Diperna1989OnTC},  Perthame \cite{PERTHAME1989191}, Gerard \cite{Gerard1991}, Golse et al\cite{Golse1988}, Bouchut et al \cite{BouchutBook}). 

Proving the existence of weak solutions requires showing that one can pass to the limit in the non-linear term $f_\e \Psi(v_{f_\e})$ on a sequence of suitable approximate solutions $f_\e.$ Using a kinetic averaging lemma of  Gerard \cite{Gerard1991} we are able to show the strong compactness of moments $v_{f_\e}$ and with that the existence of weak solutions to \eqref{intro:FK1}. The situation is slightly different with equations \eqref{intro:FK2} and \eqref{intro:FK3} because the kinetic density propagates in space by a diffusion process. However, since the direction of the diffusion $e(\theta)$ varies, we still expect improved regularity for $\theta$-moments of $f$ compared to the regularity of $f.$ This we prove in lemma \ref{lemma:average} by showing that moments belong to a fractional Sobolev space $H^{\frac{2}{7},\frac{1}{7}}_{x,t},$  locally in $(x,t)\in\mathbb{R}^2\times\mathbb{R}_+.$ Such regularity implies the strong compactness of moments and the existence of weak solutions.

Kinetic models for the motion of myxobacteria, or more generally the models of self-propelled rods with alignment,  have been considered in the framework of mean-field type interactions \cite{baskaran08, baskaran08a, degond17, degond18, Frouvelle2012} and in the Boltzmann framework \cite{Bertin2006, Hittmeir1, perepelitsa2022, murphy2023}. Most works deal with the formal derivation of macroscopic equations in various asymptotic regimes, such as singular relaxation limits. Resulting equations are typically difficult to treat analytically. For example, in some models, the macroscopic equations turn out to be hyperbolic systems in non-conservative form, \cite{degond17, Hittmeir1, murphy2023}. In these works, the mathematical analysis has been mostly restricted to space homogeneous equations, addressing questions such as the existence of solutions, phase transition, and convergence to a stationary solution in the long run. 

Finally, we mention that the interaction operators similar to \eqref{intro:vf} appear in Smoluchowski equation (Doi \cite{Doi1981}) from the theory of liquid crystals and in the Kuramoto model of interacting oscillators (Kuramoto \cite{Kuramoto1975}).

\section{Models of cell motion with nematic alignment and reversals}
\subsection{Model I: liquid crystal-type alignment and reversals}
\label{sec:derive}
We start with the agent-based model of the motion of self-propelled rods (rod-shape cells) that alignment describes as in Peruani et al. \cite{Peruani2008}. The alignment part of the model describes the change of orientation of rods according to Maier-Saupe of alignment of liquid crystals. The model is  described by the equations for positions $X_i$ and orientations $\Theta_i$ of N cells: 
\begin{align}
\frac{dX_i}{dt}{}&={}v_0e(\theta_i),\quad e(\theta)=(\cos\theta,\sin\theta),\;v_0>0,\\
\label{eq:align}
\frac{d\Theta_i}{dt}{}&={} \gamma \sum_{j\,:\,|X_j-X_i|<r}\sin(2(\Theta_j-\Theta_i)){}+{}\eta_i(t),
\end{align}
$i=1\dots,N.$ Here, $\gamma>0,$  and $\eta_i(t)$ are stochastic processes such that: (i)  $\eta_i(t)$ is uniformly distributed  on interval $[-s,s],$ where $s=\sqrt{3}\sigma,$ with $\sigma>0;$ (ii)  for all $i,j$ and $t,t',$
\[
\langle \eta_i(t)\rangle{}={}0,\quad \langle \eta_i(t)\eta_j(t')\rangle = \sigma^2\delta_{i,j}\delta_{t,t'}.
\] 
For myxobacteria, which are known to periodically change their polarity and move in the opposite direction, We add cell reversals to the model, assuming that in addition to the continuous variation according to the above equations,  each $\Theta_i$ jumps to value $\Theta_i+\pi,$ with the timing of changes following independent Poisson's processes with rate $\lambda.$ 

The summation term in the rate of alignment in \eqref{eq:align} has a natural bound based on the fact that cells occupy a finite space and cells (typically) do not overlap. To account for this, we introduce the following modification. Let  $l$ and $w$ denote cells typical length and width, then
\[
| 
\sum_{j\,:\,|X_j-X_i|<r}\sin(2(\Theta_j-\Theta_i)) |{}\leq{}\frac{\pi r^2}{lw},
\]
This inequality expresses the fact that the left-hand side can not be larger than the number of cells inside a ball of radius $r$ when cells do not overlap. This is a common experimental condition. We, therefore, introduce a cutoff function
\[
\llbracket r\rrbracket_{a}=\left\{
\begin{array}{rr}
-a, & r<-a,\\
r, & |r|\leq a,\\
a, & r>a.
\end{array}\right.
\]
and use it in the ODE for $\Theta_i$ to enforce the bound on the speed of alignment: 
\begin{equation}
\label{eq:Theta2}
\frac{d\Theta_i}{dt}{}={} \gamma \left\llbracket \sum_{j\,:\,|X_j-X_i|<r}\sin(2(\Theta_j-\Theta_i))\right\rrbracket_{\frac{\pi r^2}{lw}}{}+{}\eta_i(t).
\end{equation}
Alternatively, the rate of alignment can be written using the mean nematic current at position $X_i:$
\[
\mathbf{J}_i = \sum_{j\,:\,|X_j-X_i|<r}e^{2\i\Theta_j}
\]
and the mean nematic director $\bar{\Theta}_i\in (-\pi/2,\pi/2]$ which is defined when $\mathbf{J}_i\not=0$ as
\[
\bar{\Theta}_i{}={}\frac{1}{2}\mbox{\rm Arg}\frac{\mathbf{J}_i}{|\mathbf{J}_i|}.
\]
In this way,
\[
\sum_{j\,:\,|X_j-X_i|<r}\sin(2(\Theta_j-\Theta_i)){}={}|\mathbf{J}_i|\sin(2(\bar{\Theta}_i-\Theta_i)).
\]
Thus, a different way to limit the rate of alignment is to set it to be equal to 
\begin{equation}
\label{eq:Theta3}
\gamma \left\llbracket|\mathbf{J}_i|\right\rrbracket_{\frac{\pi r^2}{lw}}\sin(2(\bar{\Theta}_i-\Theta_i)).
\end{equation}
In this paper, we will use the expression for the nematic alignment in the form appearing in \eqref{eq:Theta2}, but all our results also apply to the model based on \eqref{eq:Theta3}.

Denoting by $F(x,\theta,t)$ the distribution function for the number of cells in $(x,\theta)$ space, one obtains a kinetic equation for evolution of $F,$
\begin{equation}
\label{eq:FK1}
\partial_t F{}+{}v_0e(\theta)\cdot\grad F{}+{}\gamma\partial_{\theta}(F\left\lb V^r_F\right\rb_{\frac{\pi r^2}{lw}}){}-{}\mu_a\partial^2_{\theta\theta}F{}={}-\lambda F + \lambda \tau_\pi F,
\end{equation}
where
\[
V^r_F(x,\theta,t){}={} \int_{|y-x|<r}\int_\T F(y,\theta_1,t)\sin(2(\theta_1-\theta))\,d\theta_1dy,\quad \T=(-\pi,\pi),
\]
and $\mu_a{}={}\frac{\sigma^2}{2}.$ Equation \eqref{eq:FK1} (without reversals and limited alignment rate) was derived in \cite{Peruani2008}. We rescale variables with some length and time scales $L$ and $T:$
\[
x=L\xb,\quad t=T\tb,\quad F(x,\theta,t) = NL^{-2}\fb(\xb,\theta,\tb),
\]
where $\fb$ is a probability density function. We write equation \eqref{eq:FK1} in new non-dimensional variables and drop bars for brevity:
\begin{equation}
\label{eq:FK2}
\partial_t f{}+{}\frac{v_0T}{L}e(\theta)\cdot\grad f{}+{}(T\gamma)(\frac{N\pi r^2}{L^2})\partial_{\theta}(f\lb V^{\frac{r}{l}}_f\rb_{\frac{L^2}{Nlw}}){}-{}T\mu_a\partial^2_{\theta\theta}f
{}={}-T\lambda f + T\lambda \tau_\pi f,
\end{equation}
with
\[
V_f^{\frac{r}{L}}(x,\theta,t){}={} \dashint_{|y-x|<\frac{r}{L}}\int_\T f(y,\theta_1,t)\sin(2(\theta_1-\theta))\,d\theta_1dy,
\]
In the first asymptotic regime that we consider in this paper, we assume that $r/L= \e\ll 1,$ while considering all other coefficients to be of the finite order. In particular, this means that $N$ is large but $Nr^2L^{-2}$ is finite. For the bacterial motion, one typically has $r\sim l,$ and $w\sim 0.1l,$ 
 see \cite{ZHANG2018}. With these relations, the coefficient in the cutoff function $V_f,$ $L^2(Nlw)^{-1}$ is of the finite order.

For notational simplicity, we set all finite order coefficients  to be equal 1 and consider the equation
\begin{equation}
\label{eq:FK3}
\partial_t f{}+{}e(\theta)\cdot\grad f{}+{}\partial_{\theta}(f\lb V_f^\e\rb_1){}-{}\partial^2_{\theta\theta}f{}={}-f + \tau_\pi f,
\end{equation}
\begin{equation}
\label{def:Vf}
V_f^\e(x,\theta,t){}={} \dashint_{|y-x|<\e}\int_\T f(y,\theta_1,t)\sin(2(\theta_1-\theta))\,d\theta_1dy.
\end{equation}
The formal $\e\to0$ limit of this equation is the equation with the interactions localized to a point:
\begin{equation}
\label{def:vf}
V_f^0{}={}\int_\T f(x,\theta_1,t)\sin(2(\theta_1-\theta))\,d\theta_1,
\end{equation}  
We will use notation $v_f$ for $V_f^0.$


\subsection{Model II: macroscopic diffusion limit}
Now we return to PDE \eqref{eq:FK2} and consider a different asymptotic regime. Let $\tau = \frac{1}{\lambda}$ be the meantime traveled by a cell between reversals.
For the motion of myxobacteria cell, one typically has
\[
l \sim v_0\tau,
\]
see for example Cotter et al. \cite{Cotter17} and Murphy et al. \cite{Murphy23}.
Thus, the macroscopic limit $l/L=\e \ll1$ can be expressed as $v_0\tau/L\sim\e \ll 1.$ As in the previous model, we assume
\[
r\sim l,\quad \frac{Nlw}{L^2}\sim 1.
\]
Microscopic reversal-type motion, at the macroscopic level, is described as diffusion along the lines parallel to the cell orientation $e(\theta),$  with the diffusion coefficient
\[
d{}={}\frac{(\tau v_0)^2}{\tau}{}={}\tau v_0^2.
\]
To observe this diffusion at the macro-level we select time scale $T$ it must be selected such that 
\[
\frac{L^2}{T}{}={}d.
\]
 This implies that
 \[
 \frac{T}{\tau}{}\sim{}\frac{1}{\e^2}.
 \]
 Concerning the alignment and orientational rates $\gamma$ and $\mu,$ we assume that
 \[
 \gamma T\sim 1,\quad  \mu T\sim 1.
 \]
In other words, we are discussing a regime with fast reversals and slow alignment and orientation diffusion.
 Using the above relations in \eqref{eq:FK2} we arrive at the equation
 \begin{equation}
\label{eq:FK5}
\partial_tf +\frac{1}{\e}e(\theta)\cdot f{}+{}\partial_\theta(f\lb V_f^\e\rb_1) - \partial_{\theta\theta}^2f{}={}\frac{1}{\e^2}\left(
-f(x,\theta,t) + f(x,\theta+\pi,t)\right),
\end{equation}
with $V_f^\e$  as in \eqref{def:Vf}.
We will also need  equation for $f(x,\theta,t)+f(x,\theta+\pi,t):$
\begin{multline}
\label{eq:FK6}
\partial_t(f(\theta)+f(\theta+\pi)){}+{}\frac{1}{\e}e(\theta)\cdot\grad(f(\theta)-f(\theta+\pi))
{}+{}\partial_\theta((f(\theta)+f(\theta+\pi))\lb V_f^\e\rb_1) \\
{}-{}\partial_{\theta\theta}^2(f(\theta)+f(\theta+\pi)){}={}0.
\end{multline}
Consider now the singular limit of $\e\to0.$
Writing $f = f_0+\e f_1 +o(\e)$  we find from \eqref{eq:FK5} that $f_0(x,\theta+\pi,t) = f_0(x,\theta,t)$ and 
\[
f_1(x,\theta+\pi,t) - f_1(x,\theta,t) = e(\theta)\cdot\grad f_0.
\]
Using these relations in \eqref{eq:FK6}, we find the leading order equation for $f_0:$
\begin{equation}
\label{eq:FK7}
\partial_t f{}-{}\div \left(e(\theta)\otimes e(\theta)\grad f\right){}+{}\partial_\theta(f \lb v_f\rb_1) {}-{}\partial^2_{\theta\theta}f{}={}0.
\end{equation}

\subsection{Model III: uni-directional chemotaxis}

A model in which reversals  are biased by a chemical signal can be described by setting the rate of reversals $\lambda$ to depend on the chemoattractant concentration $S(x,t),$ for example, as in the following equation 
\[
\lambda = \lambda_0 -\beta\Phi(l^3_{ch} e(\theta)\cdot\grad S),
\]
where $\Phi$ is increasing function with $\Phi(0)=0,$ $\Phi'(0)=1,$  and $\sup_r|\Phi(r)|<\frac{\lambda_0}{\beta}.$ Here $\beta>0$ is the chemotaxis intensity (rate), and $l_{ch}$ is the characteristic length scale of signal sensing (distance over which the changes in $S$ are sensed by a cell). $\lambda_0$ is the unbiased reversal rate.

Switching to the non-dimensional variables
\begin{equation}
\label{def:s}
S(x,t) = NL^{-2}s(x/L,t/T),
\end{equation}

and 
\[
\lambda T = \lambda_0T -  \beta T\Phi(\frac{Nl_{ch}^3}{L^3}e(\theta)\cdot s).
\]
For the asymptotic regime, in addition to assumptions made in the previous model, we postulate
\[
\beta T\sim 
\lambda_0 T\sim \frac{1}{\e^2},
\quad l_{ch}\sim l.
\]
The last assumption implies that 
\[
\frac{Nl_{ch}^3}{L^3}{}\sim{}\frac{Nl^2}{L^2}\frac{l}{L}\sim \e.
\]
Setting all finite order coefficients to 1, for simplicity, we obtain the expression for the non-dimensional reversal rate
\[
T\lambda = \frac{1}{\e^2}{}-{}\frac{1}{\e}e(\theta)\cdot\grad s + O(1).
\]
Retaining only the first two terms, we arrive at the equation
\begin{align*}
\partial_t f +\frac{1}{\e}e(\theta)\cdot\grad f +\partial_\theta (f \lb V_f^\e\rb_1)  -
\partial^2_{\theta\theta} f{}={}& -\frac{1}{\e^2}(1-\e e(\theta)\cdot\grad s)f(\theta)
\\
&{}+{}\frac{1}{\e^2}(1+\e e(\theta)\cdot\grad s)f(\theta+\pi).
\end{align*} 
The formal asymptotic limit of $\e\to0$ leads to the following equation for $\pi$-periodic in $\theta$ function $f.$
\begin{align}
\label{eq:FK9}
&\partial_t f + \div(  2[e(\theta)\grad s] e(\theta) f){}-{}\div (e(\theta)\otimes e(\theta)\grad f){}+{}\partial_\theta(f \lb v_f\rb_1) {}-{}\partial^2_{\theta\theta}f{}={}0.
\end{align}
Here, $v_f$ as in \eqref{def:vf}.


Finally, we mention a coupled system of chemotaxis, when the chemoattractant is secreted by cells and spreads by diffusion
\begin{equation}
\label{eq:S}
\partial_t S - \mu_S\Delta S = \alpha n-\frac{1}{\tau_S}S,\quad n(x,t)= \int_{\T}F(x,\theta,t)\,d\theta,
\end{equation}
where $\mu_S,$ $\alpha$, and $\tau_S$ stay for the chemoattractant diffusion coefficient, the rate of production, and the characteristic decay time (relaxation time). Assuming that the coefficients verify relations
\[
\alpha T \sim 1, \quad \frac{T}{\tau_S}\sim 1, \quad  \frac{T\mu_S}{L^2}\sim 1,
\]
the non-dimensional concentration $s,$ defined in \eqref{def:s},
verifies equation
\begin{equation}
\label{eq:s}
\partial_t s - \Delta s =  n-s,\quad n(x,t)= \int_{\T}f(x,\theta,t)\,d\theta.
\end{equation}
We postpone to the future the analysis of the coupled system \eqref{eq:FK9}, \eqref{eq:s}.


\subsection{Existence theory: model I}

In our first theorem, we prove the existence of a weak solution to the initial-value problem for equation \eqref{eq:FK3}, \eqref{def:Vf} and then we pass to the limit of $\e\to0$ to obtain solutions with alignment rate \eqref{def:vf}.

All equations are considered on domain $\mathbb{R}^2\times\mathbb{R}\times(0,+\infty),$ with solutions $f$ being periodic over $\Pi=(0,1)\times(0,1)$ in spacial variable $x$ and over $\T{}={}(-\pi,\pi)$ in $\theta$ variable. Initial conditions are given as $f(x,\theta,0)=f_0(x,\theta).$  We will use notation $L^2_{x,\theta}$ for the space $L^2(\Pi\times\T),$ with the standard inner product $\langle\cdot,\cdot\rangle_{L^2_{x,\theta}},$ and $H^1_{x,\theta}$ for $H^1(\Pi\times\T).$
Let $Q_T=\Pi\times\T\times(0,T).$

We will work with weak solutions that verify
the inequality 
\begin{equation}
\label{eq:energy1}
\sup_{[0,T]}\|f(\cdot,t)\|^2_{L^2_{x,\theta}} {}+{}\|\partial_\theta f\|^2_{L^2(0,T;L^2_{x,\theta})}{}+{}
\|f(\theta) - f(\theta+\pi)\|^2_{L^2(0,T;L^2_{x,\theta})}{}\leq{}2e^{T}\|f_0\|^2_{L^2_{x,\theta}},
\end{equation}
obtained by integrating the equation against solution $f.$

\begin{theorem}
\label{th:1}
 Let $f_0\in L^2_{x,\theta}\cap L^1_{x,\theta},$ periodic in $(x,\theta),$ $f_0\geq0$ a.e. and have unit mass.
There is a unique, global in time, periodic, weak solution of \eqref{eq:FK3}. The solution is non-negative, has unit mass for all $t>0,$ and such that for any $T>0,$
\begin{align}
\label{th1:1}
&f{}\in{} L^\infty(0,T;L^2_{x,t})\cap C([0,T]; L^2_{x,\theta}-weak)\\
\label{th1:2}
&\partial_\theta f \in L^2(0,T;L^2_{x,\theta})\\
\label{th1:3}
&\partial_t \langle f,w\rangle_{L^2_{x,\theta}} \in L^2(0,T),\quad \forall w\in H^{1}_{x,\theta},
\end{align}
and \eqref{eq:energy1} holds.
\end{theorem}
\begin{proof}
The weak solution can be constructed by a parabolic approximation, by adding  $x$-diffusion $\delta\Delta_xf$ to the equation and passing to the limit $\delta\to0.$ A theory of linear parabolic equations and a fixed point argument can be used to obtain such a solution (we omit the details).
Solutions, call them $f_\delta,$ verify estimate \eqref{eq:energy1} and thus, have a subsequence $f_{\delta_k}$ that converges to some 
$f$ -- the element of spaces listed in \eqref{th1:1}--\eqref{th1:3}, in the following topologies:
\begin{align}
&f_{\delta_k} \to f\quad  \mbox{\rm weakly in}\; L^2(0,T;L^2_{x,\theta})\\
&\partial_\theta f_{\delta_k}\to \partial_\theta f \quad  \mbox{\rm weakly in}\; L^2(0,T;L^2_{x,\theta})\\
&\langle f_{\delta_k},w\rangle_{L^2_{x,\theta}} \to \langle f,w\rangle_{L^2_{x,\theta}} \quad\mbox{\rm in}\; C([0,T]),\quad \forall w\in L^2_{x,\theta}
\end{align}
The only nonlinear term in the equation \eqref{eq:FK3} is operator
\begin{equation}
\label{def:opA}
A(f) = \partial_\theta(f \lb V_f^\e\rb_1) = \partial_\theta f\lb V_f^\e\rb_1
{}-{}2f\dashint_{|y-x|<\e}\int_\T f(y,\theta_1,t)\cos(2(\theta_1-\theta))\,d\theta_1dy\chi_{\{ |V_f^\e(x,\theta,t)|\leq 1\}}.
\end{equation}
Given the above compactness properties of $f_{\delta_k},$ $V^\e_{f_{\delta_k}}$ and the other integral term in \eqref{def:opA} converge uniformly on compact subset in $(x,\theta,t)$ domain, and consequently,
\[
A(f_{\delta_k})\to A(f),\quad \mbox{\rm weakly in}\quad L^2(0,T;L^2_{x,\theta}).
\]

For the uniqueness, given $f_1$ and $f_2$ in the regularity class of the theorem, we can estimate for any $t>0,$
\begin{equation*}
\|A(f_1(t))-A(f_2(t))\|_{L^2_{x,\theta}}{}\leq C\|\partial_\theta f_1(t)-\partial_\theta f_2(t)\|_{L^2_{x,\theta}}{}+{}
C_\e\left(\|f_1(t)\|_{L^2_{x,\theta}}
{}+{}\|\partial_\theta f_1(t)\|_{L^2_{x,\theta}}\right)\|f_1(t)-f_2(t)\|_{L^2_{x,\theta}},
\end{equation*}
for some $C>0,$ and $C_\e>0.$ The uniqueness then follows from  the ``energy" inequality for $f_1-f_2$ and the Gronwall's lemma.
\end{proof}

\begin{theorem}
\label{th:2}
Let $f_\e$ be the family solutions \eqref{eq:FK3} from theorem \ref{th:1} corresponding to initial data $f_0,$ as in theorem \ref{th:1}.
The set of moments
\[
\left\{\int_\T f_{\e} \phi(\theta)\,d\theta\right\} \quad \mbox{\rm is pre-compact in  } L^2_{loc}(\mathbb{R}^2\times(0,+\infty)), \quad \forall \phi\in C^1(\overline{\T}).
\]
On a suitable subsequence $\e_n\to0$, $f_{\e_n}$ converges to a weak solution \eqref{eq:FK3}, \eqref{def:vf}, $f,$ with properties \eqref{th1:1}--\eqref{th1:3}.
\end{theorem}
\begin{proof}
We start with weak solutions $f_{\e}$ of \eqref{eq:FK3} that verify the energy inequality \eqref{eq:energy1}.
Given the energy estimates, we only need to establish the strong compactness property of moments $V_{f_\e}^\e$, for a solution to exist. For that, we use a version of the result of Gerard \cite{Gerard1991} from Bouchut et al \cite{BouchutBook}, Theorem 1.5. We write \eqref{eq:FK3} as  a transport equation
\[
\partial_t f_{\e}{}+{}e(\theta)\cdot\grad f_{\e}{}={}-f_\e(\theta)+f_\e(\theta+\pi) - \partial_\theta(f \lb V^{\e}_{f_{\e}}\rb_1)+\partial^2_{\theta\theta}f_{\e}:=g_{\e}.
\]
 For $\phi\in C^1([-\pi,\pi]),$ one has $\int_{-\pi}^\pi g_{\e}\phi\,d\theta$ bounded in $L^2(0,T;L^2(\mathbb{R}^2))$ uniformly in $\e.$ This implies that
 \[
 \left\{ \int_\T g_{\e}\phi\,d\theta\right\} \quad \mbox{\rm is pre-compact in }\; H^{-1}_{loc}(\mathbb{R}^2\times\mathbb{R}_+).
 \]
 Moreover, $f_{\e}$ are bounded in $L^2(0,T;L^2_{x,v})$ for any $T>0,$ and the transport velocity verifies the property
\[
\mbox{\rm meas}\left\{ \theta\,:\, e(\theta)\cdot \xi {}={}u\right\}{}={}0,\quad \forall \xi\in \mathbb{R}^2,\,|\xi|=1,\;\forall u\in\mathbb{R}.
\]
These properties imply the compactness of the moments of $f_{\e}$ stated in the theorem.

Based on that, there a sequence $\e_n$ and $f$ such that for any $T>0,$ $f_{\e_n}\to f$ in the weak topology of $L^2(0,T;L^2_{x,v}),$ and 
\[
\int_\T f_{\e_n}(x,\theta_1,t)\sin(2(\theta_1-\theta))\,d\theta_1 \to
\int_\T f(x,\theta_1,t)\sin(2(\theta_1-\theta))\,d\theta_1
\]
a.e. $(x,\theta,t)$ and in $L^2_{loc}(Q_T).$ A subsequence can be found such that,
\[
V^{\e_n}_{f_{\e_n}}{}={}\dashint_{|y-x|<\e_n}\int_\T f_{\e_n}(y,\theta_1,t)\sin(2(\theta_1-\theta))\,d\theta_1dy\to \int_\T f(y,\theta_1,t)\sin(2(\theta_1-\theta))\,d\theta_1=v_f,
\]
a.e. $(x,\theta,t)$ and in $L^2_{loc}(Q_T).$ With this, one can pass to the limit in the equation \eqref{eq:FK3} to obtain the equation with the localized alignment rate \eqref{def:vf}.
  \end{proof}

\subsection{Existence theory: model II}
Solutions to this model will be obtained from a model with added diffusion in $x$-direction,
\begin{equation}
\label{eq:FK8}
\partial_t f{}-{}\div \left(e(\theta)\otimes e(\theta)\grad f\right){}+{}\partial_\theta(f \lb v_f\rb_1) {}-{}\partial^2_{\theta\theta}f-\e\Delta f{}={}0,
\end{equation}
in the zero limit of the diffusion coefficient $\e.$ For the solutions of the later equation, one has the energy inequality
\begin{equation}
\label{eq:energy2}
\sup_{[0,T]}\|f(\cdot,t)\|^2_{L^2_{x,\theta}} {}+{}\|\partial_\theta f\|^2_{L^2(0,T;L^2_{x,\theta})}{}+{}\| e(\theta)\cdot\grad f\|^2_{L^2(0,T;L^2_{x,\theta})}
{}+{}\e\|\grad f\|^2_{L^2(0,T;L^2_{x,\theta})}
{}\leq{}2e^{T}\|f_0\|^2_{L^2_{x,\theta}}.
\end{equation}
We will assume that weak solutions with this regularity exist for any non-negative $f_0\in L^2_{x,v}\cap L^1_{x,v},$ $\Pi$-periodic in $x$ and $\pi$-periodic in $\theta,$ and with properties \eqref{th1:1}-\eqref{th1:3}. This can be proved by the argument parallel to that of the proof of Theorem \ref{th:1}. We omit the proof.

Our main result is the following variant of the averaging lemma proved in \cite{Golse1988}.
\begin{lemma}
\label{lemma:average}
 Let $f_\e$ be a weak solution of \eqref{eq:FK8} with $\e\in(0,1],$ with initial data $f_0.$ Let $u_\e(x,t) = \psi(x,t)f_\e,$ for $\psi\in C^\infty_0(\mathbb{R}^2\times\mathbb{R}_+),$ and be extended by zero to all of $\mathbb{R}^3.$
Denote  a $\theta$-moment function
\[
\rho^\phi_\e{}={}\int_\T u_\e \phi(\theta)\,d\theta,\quad \phi\in C^1(\overline{\T}).
\]
Then, 
$
\rho^\phi_\e \in {H}^{\frac{2}{7},\frac{1}{7}}(\mathbb{R}^2\times\mathbb{R}),
$
and 
\[
\|\rho^\phi_\e\|_{{H}^{\frac{2}{7},\frac{1}{7}}}{}\leq{} C_{\psi,\phi}\|f_0\|_{L^2_{x,\theta}},
\]
for some $C_{\psi,\phi},$ independent of $\e.$ In particular,  set $\{\rho^\phi_\e\}_{\e\in(0,1]}$ is pre-compact in $L^2(\mathbb{R}^3).$
\end{lemma}
\begin{remark}
${H}^{\frac{2}{7},\frac{1}{7}}$ estimate for $\rho^\phi_\e$  of the lemma equally applies to a family of solutions of \eqref{eq:FK7}  verifying the energy balance \eqref{eq:energy2}.
\end{remark}
\begin{proof}
Let $\psi$ and $u_\e$ be as in the statement of the lemma. We will write $u$ for $u_\e$ to simplify the notation. $u$ verifies equation
\begin{equation}
\partial_t u{}-{}\div \left(D(\theta)\grad u\right){}-\e\Delta u{}={}	q_1 + \partial_\theta q_2 ,
\end{equation}
where
\[
q_1{}={}-(e(\theta)\cdot\grad f_\e)e(\theta)\cdot\grad\psi
+2\e\grad f_\e\cdot\grad u + \e f_\e\Delta\psi,
\]
\[
q_2{}={}-(u \lb v_{f_\e}\rb_1) {}+{}\partial_{\theta}u,
\]

with $q_1,$ $q_2$ uniformly in $\e$ bounded in $L^2(\mathbb{R}^2\times\T\times\mathbb{R})$ due to \eqref{eq:energy2}.
For function $\hat{u}(\xi,\theta,\tau)$ -- the Fourier transform of $u$ in $(x,t),$ we get
\[
(-i\tau + (e(\theta)\cdot\xi)^2 +\e|\xi|^2)\hat{u}{}={}
 \hat{q}_1+\partial_\theta \hat{q}_2.
\]
Writing $\xi = |\xi|e(\theta_1),$ for the polar angle $\theta_1,$ we can express
$e(\theta)\cdot\xi{}={}|\xi|\cos(\theta_1-\theta).$
Let $\omega(\theta)$ be a smooth cutoff function for the set 
\[
\{\theta\,:\, |\cos(\theta_1-\theta)|<\delta\},
\]
for small $\delta>0,$ that is,
\[
\omega(\theta){}={}\left\{
\begin{array}{ll}
0, & |\cos(\theta_1-\theta)|\leq\delta/2\\
1, & |\cos(\theta_1-\theta)|\geq\delta
\end{array}
\right.
\]
with $\max_\theta|\omega| = 1,$ and $\max_{\theta}|\partial_\theta \omega| \leq C\delta^{-1}.$
Then, we can write
\begin{align*}
\left( \int_{\T} \hat{u}\phi\,d\theta\right)^2
{}\leq{}&
2C\left(
\int_{\mathbb{T}}|\hat{u}| |1-\omega|\,d\theta
\right)^2
{}+{}
2\left(
\int_{\mathbb{T}}\frac{(\hat{q}_1+\partial_\theta \hat{q}_2)\phi(\theta)\omega(\theta)\,d\theta}{-i\tau + (|\xi|\cos(\theta_1-\theta))^2 +\e|\xi|^2}
\right)^2\\
{}\leq{}&
2C\delta
\int_{\mathbb{T}}|\hat{u}|^2 \,d\theta
{}+{}
C\int_{\T}\frac{|\hat{q}_1|^2\,d\theta}{|-i\tau + \delta^2|\xi|^2+\e|\xi|^2|^2}\\
&{}+{}
C\int_{\T}\frac{|\hat{q}_2|^2 |\partial_\theta(\phi\omega)|^2  \,d\theta}{|-i\tau + \delta^2|\xi|^2+\e|\xi|^2|^2}{}+{}
C\int_{\T}\frac{|\hat{q}_2|^2 |\xi|^4\cos^2(\theta_1-\theta)   \,d\theta}{|-i\tau + (|\xi|\cos(\theta_1-\theta))^2+\e|\xi|^2|^4}\\
{}\leq{}&
2C\delta
\int_{\mathbb{T}}|\hat{u}|^2 \,d\theta
{}+{}
C\delta^{-2}\frac{1}{(\tau^2 + \delta^4|\xi|^4)}\int_\T |\hat{q}_1|^2+|\hat{q}_2|^2\,d\theta
\\
&{}+{}
C\frac{|\xi|^2}{(\tau^2 + \delta^4|\xi|^4)^{3/2}}\int_\T |\hat{q}_2|^2\,d\theta\\
{}\leq{}&
2C\delta
\int_{\mathbb{T}}|\hat{u}|^2 \,d\theta
{}+{}
C\delta^{-6}\frac{1}{(\tau^2 + |\xi|^4)}\int_\T |\hat{q}_1|^2+|\hat{q}_2|^2\,d\theta.
\end{align*}
Choosing 
\[
\delta = (\tau^2 + |\xi|^4)^{1/7},
\]
we find that  
\[
(\tau^2 + |\xi|^4)^{1/7}\left( \int_{\T} \hat{u}\phi\,d\theta\right)^2{}\leq{}
C\left( 
\int_{\mathbb{T}}|\hat{u}|^2 {}+{} |\hat{q}_1|^2+|\hat{q}_2|^2\,d\theta
\right),
\]
with $C$ independent of $\e.$ We also have
\[
\left( \int_{\T} \hat{u}\phi\,d\theta\right)^2{}\leq{}
C 
\int_{\mathbb{T}}|\hat{u}|^2\,d\theta.
\]
Then, we obtain the estimate of the lemma, by integrating the last two inequalities in $(\xi,\tau)$ and using the Parseval's inequality.

\end{proof}

With this lemma, we can now prove 
\begin{theorem}
\label{th:3}
 Let $f_0\in L^2_{x,\theta}\cap L^1_{x,\theta},$ $f_0\geq0$ a.e., $\pi$-periodic in $\theta$ and have unit mass.
There is a global in time, weak solution of \eqref{eq:FK7}. The solution is non-negative, $\pi$-periodic in $\theta,$ has unit mass for all $t>0,$  and such that for any $T>0,$ \eqref{eq:energy2} and \eqref{th1:1}--\eqref{th1:3} hold.
\end{theorem}
\begin{remark}
The same result applies to equation \eqref{eq:FK9} with given chemoattractant concentration $s,$ when $s\in L^\infty(0,T;W^{1,\infty}(\Pi)).$
\end{remark}
\begin{proof}
We start with solutions $f_\e$ of \eqref{eq:FK8} with initial data $f_0,$ which verify the energy inequality \eqref{eq:energy2} and properties \eqref{th1:1}--\eqref{th1:3}.
A sequence $f_{\e_n},$ $\e_n\to0$ can be found such that $f_{\e_n}\to f$ in $
C([0,T]; L^2_{x,\theta}-weak).$ To pass to the limit in the nonlinear term of the equation, we use lemma \ref{lemma:average}. According to it, there is a subsequence, that we still label $\e_n,$ such that $v_{f_{\e_n}}$ is pre-compact in $L^2_{loc}(\mathbb{R}^2\times\mathbb{R}_+)$ and converges to $v_f.$ A further subsequence can be extracted such that 
$v_{f_{\e_n}}\to v_f$ a.e. $(x,t).$ With that, one can pass to the limit in the alignment term of the equation. 
\end{proof}

\bibliographystyle{plain}
\bibliography{refs}

\end{document}